\documentclass[11pt,a4paper]{article}
\usepackage[latin1]{inputenc}
\usepackage{amsmath}
\usepackage{amsfonts}
\usepackage{amssymb}
\usepackage{graphicx}
\usepackage{amsthm}
\newtheorem{proposition}{Proposition}
\newtheorem{corollary}{Corollary}
\newtheorem{example}{Example}
\newtheorem{lemma}{Lemma}
\theoremstyle{definition}
\newtheorem{remark}{Remark}
\author{Chance Sanford}
\title{Two Neumann Series Expansions for the Sine and Cosine Integrals}
\begin{document}

\maketitle

\begin{abstract}
In this work, series expansions in terms of Bessel functions of the first kind are given for the sine and cosine integrals.  These representations differ from many of the known Neumann-type series expansions for the sine and cosine integrals, which contain non-integer order or quadratic Bessel function terms.  In addition, using the theory of Euler sums we are able to obtain some closed form evaluations of integrals involving Bessel functions and the sine and cosine integrals.
\end{abstract}

\section*{Inroduction and Statement of Results}

\subsection*{Introduction}
The objective of this note is to present a new series expansion involving Bessel functions for both the Sine and Cosine Integrals.  In contrast to the majority of the known expansions for the Sine and Cosine Integrals involving Bessel functions, the series presented below contain linear terms of integer-order Bessel functions of the First Kind. \\

These expansions are termed \emph{Neumann Series} \cite{watson,wilkins}; that is, series of the form

\begin{equation}
f(x) = \sum_{n=0}^{\infty}a_{n}J_{2n+v+1}(x) \label{intro_eq}
\end{equation}

where $v \in \{-1,0,1,2,\dots\}$ and $J_v(z)$ is the Bessel function of the First Kind defined by 

\begin{equation*}
J_v(z) := \sum_{n=0}^{\infty}\frac{(-1)^n}{n!\Gamma(v+n+1)}\left(\frac{z}{2}\right)^{2n+v}
\end{equation*}

For ease of reference we will refer to the coefficients $a_n$ in \eqref{intro_eq} as the \textit{Neumann coefficients} of $f(x)$.

Previously, other authors have considered similar series expansions involving Bessel functions for the Sine and Cosine integrals.  In particular, Nielsen in \cite{nielsen1} (or \cite{nielsen2}) provides the following expansion for the variant Sine integral in terms of the half-integer-order Bessel functions of the First Kind

\begin{equation*}
\mathrm{si}(x) = -\frac{\pi}{2}\sum_{n=0}^{\infty}(-1)^n\frac{J_{\frac{n}{2}}(x)}{\Gamma(\frac{n}{2}+1)}\left(\frac{x}{2}\right)^{\frac{n}{2}}
\end{equation*}

as well as the similar expansion

\begin{equation*}
\mathrm{Si}(x) = \frac{\pi}{2}\sum_{n=0}^{\infty}\frac{J_{n+\frac{1}{2}}(x)}{\Gamma(n+\frac{3}{2})}\left(\frac{x}{2}\right)^{n+\frac{1}{2}}
\end{equation*}

More recently Harris \cite{harris}, searching for the Cosine Integral analogue to the Spherical Bessel function expansion

\begin{equation}
\mathrm{Si}(2x) = 2x\sum_{n=0}^{\infty}\left[j_n(x)\right]^2 \label{eq:I1}
\end{equation}

stated in \cite{abramowitz/stegun}, showed that

\begin{equation}
\mathrm{Ci}(2x) = \gamma + \log(2x) +\sum_{n=0}^{\infty}a_n\left[j_n(x)\right]^2 \label{eq:I2}
\end{equation}

where $a_n$ is defined by $a_0 = 0$ and

\begin{equation*}
a_n = -(2n+1)\left[1-(-1)^n + \sum_{j=1}^{n}\frac{1}{j}\right] \quad (n>0)
\end{equation*}

Lastly, Luke in \cite{luke} recorded over a dozen series expansions involving Bessel functions of the First Kind for $\mathrm{Si}(x)$ and $\mathrm{Ci}(x)$, including the following Cosine Integral expansion

\begin{equation*}
\mathrm{Ci}(x) = \gamma + \log(x) - \frac{1}{2}\sum_{n=0}^{\infty}\frac{\left[\psi\left(n+\tfrac{1}{2}\right)-\psi\left(\tfrac{1}{2}\right)\right]}{n!}\left(\frac{x}{2}\right)^n J_n(x)
\end{equation*}

where $\psi(x)$ is the Digamma function.\\

Before proceeding on to our main results, we quickly review the notation and definitions used throughout the paper for the Sine and Cosine Integrals.\\

\subsection*{Notation and Definitions}

To begin, we have the Sine Integral defined by

\begin{align*}
\mathrm{Si}{(x)} &:= \int_{0}^{x}{\frac{\sin{t}}{t}\text{d}t} \\
&= \sum_{n=1}^{\infty}\frac{(-1)^{n-1} x^{2n-1}}{(2n-1)(2n-1)!}
\end{align*}

while the variant Sine Integral $\text{si}(x)$ (found for instance in Gradshteyn and Ryzhik \cite{GR}) is related to $\text{Si}(x)$ by the identity

$$ \text{si}{(x)} := -\int_{x}^{\infty}{\frac{\sin{t}}{t}\text{d}t} = \text{Si}{(x)} - \frac{\pi}{2} $$

The Cosine Integral function on the other hand is defined by

\begin{align*}
\text{Ci}{(x)} &:= -\int_{x}^{\infty}{\frac{\cos{t}}{t}\text{d}t} \\
&:= \gamma + \log{x} + \int_{0}^{x}{\frac{\cos{t}-1}{t}\text{d}t} \\
&= \gamma + \log{x} + \sum_{n=1}^{\infty}\frac{(-1)^{n-1} x^{2n}}{2n(2n)!}
\end{align*}

where $\gamma \approx 0.577215\dots$ is Euler's Constant.\\

In addition, before moving on to the main results of this paper, we present some notation for the  harmonic numbers, Riemann zeta function, as well as Euler sums.

The harmonic numbers and their alternating analogue are defined respectively as

\begin{equation*}
H_n = \sum_{k=1}^{n}\frac{1}{k}, \qquad A_n = \sum_{k=1}^{n}\frac{(-1)^{k-1}}{k}
\end{equation*}

while the Riemann zeta function and the Dirichlet eta or alternating zeta function are defined respectively as

\begin{equation*}
\zeta(s) = \sum_{k=1}^{\infty}\frac{1}{k^s}, \qquad \eta(s) = \sum_{k=1}^{\infty}\frac{(-1)^{k-1}}{k^s}
\end{equation*}

Finally, in the second half of this work we will encounter the related Euler sums. These functions are, in their most basic configuration, series of the form

\begin{equation*}
\sum_{n=1}^{\infty}(\pm 1)^n\frac{H^m_n}{n^k}, \qquad \sum_{n=1}^{\infty}(\pm 1)^n\frac{A^m_n}{n^k}
\end{equation*}

where $H_n$ and $A_n$ are the harmonic and alternating harmonic numbers described above.

There now exists an enormous body of literature involving Euler sums and their generalizations.  For more information into this fascinating subject the website \cite{hoffman} provides as complete of a bibliography as one could wish for.

\subsection*{Main Results}

We now present our main results, beginning with the following expansion for the Sine Integral. 

%Proposition 1
\begin{proposition}
For $a \geq 0$,
\begin{equation}
\mathrm{Si}(a) = 2\sum_{n=0}^{\infty} J_{2n+1}(a)\alpha_n \label{eq:P1}
\end{equation}

where the coefficients $\alpha_n$ are given by

\begin{equation*}
\alpha_n = 2\sum_{k=1}^{n}\frac{(-1)^{k-1}}{2k-1} + \frac{(-1)^n}{2n+1}
\end{equation*}

\end{proposition}

%Corollary 1 
\begin{corollary}
For $n \geq 0$,
\begin{align}
\int_{0}^{\infty}\mathrm{Si}(t)J_{2n+1}(t)\frac{\mathrm{d}t}{t} &= \frac{1}{2n+1}\int_{0}^{\pi/2}\sin[(2n+1)t]\cot(t)\mathrm{d}t \label{eq:C1.0} \\ &= \frac{\alpha_n}{2n+1} \label{eq:C1.1} \\ 
&= \sum_{k=0}^{n}\frac{(n+k)!}{(n-k)!}\frac{2^{2k}(-1)^k}{(2k+1)(2k+1)!} \label{eq:C1.2}
\end{align}

\end{corollary}

%Proposition 2
\begin{proposition}
For $a \geq 0$, 
\begin{equation}
\mathrm{Ci}(a) = \gamma + \log(a) - 2\sum_{n=1}^{\infty}J_{2n}(a)\beta_n \label{eq:P2}
\end{equation}

where the coefficients $\beta_n$ are given by

\begin{equation*}
\beta_n = H_n + A_n - \frac{1}{2n} - \frac{(-1)^{n-1}}{2n}
\end{equation*}

\end{proposition}

%Corollary 2
\begin{corollary}

For $n \geq 1$,
\begin{align}
\int_{0}^{\infty}\left[\gamma+\log(t)-\mathrm{Ci}(t)\right]J_{2n}(t)\frac{\mathrm{d}t}{t} &= \frac{1}{2n}\int_{0}^{\pi/2}\left[1-\cos(2nt)\right]\cot(t)\mathrm{d}t \label{eq:C2.0} \\ &= \frac{\beta_n}{2n} \label{eq:C2.1}
\\ &= \sum_{j=0}^{n-1}\frac{(-1)^j 2^{2j}}{(j+1) (2j+2)!}\frac{(n+j)!}{(n-j-1)!} \label{eq:C2.2}
\end{align}

\end{corollary}

\section*{Proof of Results}

We begin by establishing an integral representation for the Neumann coefficients of $\tfrac{1}{2}\mathrm{Si}(x)$.

\begin{lemma}
For $n \in \{0,1,2,\dots\}$, the following identity holds:
\begin{equation}
\int_{0}^{\pi/2}\sin[(2n+1)t]\cot(t)\mathrm{d}t = 1-2\sum_{k=1}^{n}\frac{(-1)^k}{4k^2-1}
\end{equation}

\begin{proof}
The proof of Lemma 1 begins with the trigonometric identity

\begin{equation}
\frac{\sin\left[\left(2n+1\right)t\right]}{\sin(t)} = 1 + 2\sum_{k=1}^{n}\cos(2kt) \label{eq:pL1.1}
\end{equation}

which can be established by summing over the product-to-sum formula

\begin{equation*}
2\sin(t)\cos(2kt) = \sin\left[\left(2k+1\right)t\right] - \sin\left[\left(2k-1\right)t\right]
\end{equation*}

then dividing by $\sin(t)$.\\

Proceeding on, we multiply \eqref{eq:pL1.1} by $\cos(t)$ and integrate over $\left[0,\tfrac{\pi}{2}\right]$, which produces

\begin{equation*}
\int_{0}^{\pi/2}\sin[(2n+1)t]\cot(t)\mathrm{d}t = \int_{0}^{\pi/2}\cos(t)\mathrm{d}t + 2\int_{0}^{\pi/2}\cos(t)\sum_{k=1}^{n}\cos(2kt)\mathrm{d}t
\end{equation*}

After interchanging the order of integration and summation on the right-hand side and carrying out the integration, using along the way the elementary integral

\begin{equation*}
\int_{0}^{\pi/2}\cos(t)\cos(2kt)\mathrm{d}t = -\frac{\cos(\pi k)}{4k^2-1}
\end{equation*}

we discover that

\begin{equation*}
\int_{0}^{\pi/2}\sin[(2n+1)t]\cot(t)\mathrm{d}t = 1-2\sum_{k=1}^{n}\frac{(-1)^k}{4k^2-1}
\end{equation*}

Breaking up the sum on the RHS gives

\begin{align*}
1+2\sum_{k=1}^{n}\frac{(-1)^{k-1}}{4k^2-1} &= 1 + \sum_{k=1}^{n}(-1)^{k-1}\left(\frac{1}{2k-1} - \frac{1}{2k+1}\right) \\
&= 1 + \sum_{k=1}^{n}\frac{(-1)^{k-1}}{2k-1} + \sum_{k=2}^{n+1}\frac{(-1)^{k-1}}{2k-1} \\
&= 2\sum_{k=1}^{n}\frac{(-1)^{k-1}}{2k-1} + \frac{(-1)^{n}}{2n+1}
\end{align*}

as desired, which enables us to move on to the proof of Proposition 1.

\end{proof}

\end{lemma}

\begin{proof}[Proof of Proposition 1]

Formula 8.518.6 in \cite{GR} states that

\begin{equation}
\sum_{n=0}^{\infty}J_{2n+1}(a)\sin(2n+1)t = \frac{\sin(a\sin(t))}{2} \label{eq:pP10.1}
\end{equation}

After multiplying each side of \eqref{eq:pP10.1} by $\cot(t)$, then integrating with respect to $t$ over $[0,\frac{\pi}{2}]$ we obtain

\begin{equation*}
\sum_{n=0}^{\infty}J_{2n+1}(a)\int_{0}^{\pi/2}\sin[(2n+1)t]\cot(t)\mathrm{d}t = \frac{1}{2}\mathrm{Si}(a) \label{eq:pP10.2}
\end{equation*}

where we have justified the interchange of summation and integration by the absolute convergence of the series, and used

\begin{equation*}
\mathrm{Si}(a) =\int_{0}^{\pi/2}\sin(a\sin{t})\cot(t)\mathrm{d}t
\end{equation*}

which may be quickly established with the change of variable $z = \sin(t)$.
Employing Lemma 1 then completes the proof.

\begin{align*}
\frac{1}{2}\mathrm{Si}(a) &= \sum_{n=0}^{\infty}J_{2n+1}(a)\int_{0}^{\pi/2}\sin[(2n+1)t]\cot(t)\mathrm{d}t \\
&= \sum_{n=0}^{\infty} J_{2n+1}(a)\alpha_n
\end{align*}

\end{proof}

In \cite{wilkins} Wilkins showed that for a function $f(x)$ possessing a Neumann series expansion of the form 

\begin{equation}
f(x) = \sum_{n=0}^{\infty}a_{n}J_{2n+v+1}(x) \quad (v > -1) \label{eq:Expansion}
\end{equation}

there exists an integral representation for the Neumann coefficients $a_{n}$. We present this result as the following lemma.

\begin{lemma}[\textbf{Wilkins}]
For a function $f(x)$ possessing an expansion as in \eqref{eq:Expansion} the following representation for the Neumann coefficients of $f(x)$ holds

\begin{equation}
a_{n} = 2(2n+v+1)\int_{0}^{\infty}f(t)J_{2n+v+1}(t)\frac{\mathrm{d}t}{t}
\end{equation}

In addition, if $f(x)$ possesses a power series expansion of the form $x^{-v}f(x) = \sum_{n=0}^{\infty}b_{n} x^{2n+1}$, then we have

\begin{equation}
a_{n} = 2(2n+v+1)\sum_{k=0}^{n}\frac{\Gamma(n+k+v+1)}{(n-k)!}2^{2k+v}b_{n}
\end{equation} 

\end{lemma}

\begin{proof}[Proof of Corollary 1]

From Lemma 2 we saw that the Neumann coefficients of $\tfrac{1}{2}\mathrm{Si}(x)$ are

\begin{equation*}
a_n =\int_{0}^{\pi/2}\sin[(2n+1)t]\cot(t)\mathrm{d}t = \alpha_n
\end{equation*}

consequently \eqref{eq:C1.0} and \eqref{eq:C1.1} immediately follow from the preceding lemma. \\

Likewise by the series definition for $\mathrm{Si}(x)$ and the above identity, the validity of \eqref{eq:C1.2} quickly follows.

\end{proof}

Before continuing on with the proof of Proposition 2, we first provide an integral for the Neumann coefficients in \eqref{eq:P2}.

%Lemma 3
\begin{lemma}
For $n \in \mathbb{N}_0$ the following identity holds:
\begin{equation}
\int_{0}^{\pi/2}\left[1-\cos(2nt)\right]\cot(t)\mathrm{d}t = H_n + A_n - \frac{1}{2n} - \frac{(-1)^{n-1}}{2n}
\end{equation}

\begin{proof}

As in the proof of Lemma 1, we begin with the identity

\begin{equation}
\frac{1-\cos(2nt)}{\sin(t)} = 2\sum_{k=1}^{n}\sin\left[(2k-1)t\right] \label{eq:pL2.1}
\end{equation}

which may be derived using

\begin{equation*}
2\sin(t)\sin\left[(2k-1)t\right] = \cos\left[\left(2k-2\right)t\right] - \cos(2kt)
\end{equation*}

By multiplying \eqref{eq:pL2.1} by $\cos(t)$  and integrating over $\left[0,\tfrac{\pi}{2}\right]$, we arrive at

\begin{equation*}
\int_{0}^{\pi/2}\left[1-\cos(2nt)\right]\cot(t)\mathrm{d}t = 2\int_{0}^{\pi/2}\cos(t)\sum_{k=1}^{n}\sin\left[(2k-1)t\right]\mathrm{d}t
\end{equation*}

Then, after interchanging the order of summation and integration, and using the identity

\begin{equation}
2\int_{0}^{\pi/2}\cos(t)\sin\left[(2k-1)t\right]\mathrm{d}t = \frac{1 - \cos(\pi k)}{2k} + \frac{1 + \cos(\pi k)}{2k-2} \label{eq:pL2.2}
\end{equation}

we come to the following representation

\begin{equation*}
\int_{0}^{\pi/2}\left[1-\cos(2nt)\right]\cot(t)\mathrm{d}t = \sum_{k=1}^{n}\frac{1-(-1)^k}{2k} + \sum_{k=1}^{n}\frac{1+(-1)^k}{2k-2}
\end{equation*}

which after some rearrangement gives

\begin{align*}
	\int_{0}^{\pi/2}\left[1-\cos(2nt)\right]\cot(t)\mathrm{d}t &= \sum_{k=1}^{n}\frac{1}{k} + \sum_{k=1}^{n}\frac{(-1)^{k-1}}{k} - \frac{1}{2n} - \frac{(-1)^{n-1}}{2n} 
\end{align*}

as desired. \\

With the proof of Lemma 2 complete, we move on to the proof of Proposition 2.

\end{proof}

\end{lemma}

\begin{proof}[Proof of Proposition 2]

Formula 8.514.5 in \cite{GR} states that

\begin{equation}
\cos(a\sin{t}) = J_0(a) + 2\sum_{n=1}^{\infty}J_{2n}(a)\cos(2nt)
\end{equation}

Letting $t=0$ gives

\begin{equation*}
1 = J_0(a) + 2\sum_{n=1}^{\infty}J_{2n}(a)
\end{equation*}

Therefore

\begin{equation}
1-\cos(a\sin{t}) = 2\sum_{n=1}^{\infty}J_{2n}(a)\left[1-\cos(2nt)\right] \label{eq:pP2.1}
\end{equation}

Multiplying each side of \eqref{eq:pP2.1} by $\cot(t)$ and integrating with respect to $t$ over $\left[0,\tfrac{\pi}{2}\right]$ now gives

\begin{equation*}
\gamma+\log(a)-\mathrm{Ci}(a) = 2\sum_{n=1}^{\infty}J_{2n}(a)\int_{0}^{\pi/2}\left[1-\cos(2nt)\right]\cot(t)\mathrm{d}t \label{pP2.2}
\end{equation*}

where we have interchanged the order of summation and integration on the right hand side, and employed the following integral on the left

\begin{equation}
\int_{0}^{\pi/2}\left[1-\cos(a\sin{t})\right]\cot(t)\mathrm{d}t = \gamma+\log(a)-\mathrm{Ci}(a) \label{eq:pP2.3}
\end{equation}
which may be proven with a change of variable. 
 
Appealing to Lemma 3 to deal with the integral in the summand of \eqref{eq:pP2.3}, we find that

\begin{align*}
\mathrm{Ci}(a) &= \gamma + \log(a) - 2\sum_{n=1}^{\infty}J_{2n}(a)\int_{0}^{\pi/2}\left[1-\cos(2nt)\right]\cot(t)\mathrm{d}t \\
&= \gamma + \log(a) - 2\sum_{n=1}^{\infty}J_{2n}(a)\beta_n
\end{align*}

which completes the proof.
  
\end{proof}

\begin{proof}[Proof of Corollary 3]
To prove equations \eqref{eq:C2.0} and \eqref{eq:C2.1} we are forced to use a different method than Wilkins' result of Lemma 2, since \eqref{eq:pL2.1} is only valid for $v > -1$.  In \cite{wilkins} Wilkins does investigate the case where $v=-1$ and in Theorem 1.2 gives criteria which if met allows one to prove that

\begin{equation*}
a_{n} = 4n\int_{0}^{\infty}f(t)J_{2n}(t)\frac{\mathrm{d}t}{t}
\end{equation*}

In this instance we choose to forgo Wilkins theorem and take a more direct approach; presenting a proof of Corollary 2 via the \textit{method of brackets}.

The method of brackets was developed by I. Gonzales and others as an extension of Ramanujan's Master Theorem, to help tackle integrals associated with Feynman diagrams, and has proved to be a powerful tool for solving certain classes of integrals.  It should be noted that the method is still under development and at this point is still partially heuristic in nature.  For more information on the method see the recent paper \cite{gonzalez} and the references therein.

We begin by outlining the method of brackets.  A \textit{bracket} is the formal symbol $\langle a \rangle$ associated with the divergent integral

\begin{equation*}
\langle a \rangle = \int_{0}^{\infty}x^{a-1}\mathrm{d}x 
\end{equation*}

Then for a function $f(x)$ given by the formal power series

\begin{equation*}
f(x) = \sum_{n=0}^{\infty}a_n x^{\alpha n + \beta - 1}
\end{equation*}

the improper integral of $f(x)$ over the half line is formally written as the bracket series

\begin{equation*}
\int_{0}^{\infty}f(x)\mathrm{d}x  = \sum_{n=0}^{\infty}a_n\langle{\alpha n + \beta}\rangle
\end{equation*}

In addition, for convenience we introduce the symbol

\begin{equation*}
\phi_n = \frac{(-1)^n}{\Gamma(n+1)}
\end{equation*}

called the \textit{indicator of n}.

Lastly, a series of brackets is given a value according to the rule

\begin{equation}
\sum_{n=0}^{\infty}\phi_n f(n)\langle{a n + b}\rangle = \frac{1}{|a|}f(n^*)\Gamma(-n^*) \label{bracket rule}
\end{equation}

where $n^*$ is the solution to the equation $an+b=0$.  For our proof this rule is the only one needed and is essentially just Ramanujan's Master Theorem, but a generalization of this rule to multi-indexed sums, as well as others may be found in the reference cited above.

Now we proceed on to the proof of Corollary 2.  Let us first write the cosine integral in its hypergeometric form

\begin{align}
\gamma+\log(t)-\mathrm{Ci}(t) &= \frac{t^2}{4}{}_2 F_3\Biggl[\begin{array}{@{}c@{}c@{}c@{}c@{}c@{}c@{}}
& 1, & & 1 &\\
2, & & 2, & & \tfrac{3}{2} \\
\end{array}\:\biggr\vert \, {-\frac{t^2}{4}}\Biggr] \\
&= \frac{t^2}{4}\sum_{n=0}^{\infty}\frac{(1)_n(1)_n(-1)^n}{(2)_n (2)_n (\frac{3}{2})_n}\frac{t^{2n}}{2^{2n}n!}
\end{align}

where ${}_p F_q$ is the generalized hypergeometric function defined by

\begin{equation*}
{}_p F_q\Biggl[\begin{array}{@{}c@{}c@{}c@{}c@{}c@{}c@{}}
a_1 & & \dots & & a_p \\
b_1 & & \dots & & b_q\\
\end{array}\:\biggr\vert \, {z} \, \Biggr] = \sum_{n=0}^{\infty}\frac{(a_1)_n\cdots(a_p)_n}{(b_1)_n \cdots (b_q)_n}\frac{t^{n}}{n!}
\end{equation*}

and the Pochhammer symbol $(a)_n = \Gamma(n+a)/\Gamma(a)$.

Evidently we have

\begin{multline}
t^{-1}\left[\gamma+\log(t)-\mathrm{Ci}(t)\right]J_{2n}(t) = \left(\frac{t}{4}\sum_{j=0}^{\infty}\frac{(1)^2_j2^{-2j}}{(2)^2_j (\frac{3}{2})_j}\frac{(-1)^j}{j!}t^{2j}\right) \\ \times \left(\frac{t^{2n}}{2^{2n}}\sum_{k=0}^{\infty}\frac{2^{-2k}}{\Gamma(2n+k+1)}\frac{(-1)^k}{k!}t^{2k}\right)
\end{multline}

Using the elementary identity for exponential generating functions

\begin{equation*}
\left(\sum_{j=0}^{\infty}a_j\frac{x^j}{j!}\right)\left(\sum_{k=0}^{\infty}b_k\frac{x^k}{k!}\right) = \sum_{j=0}^{\infty}\left(\sum_{k=0}^{j}\binom{j}{k}a_k b_{j-k}\right)\frac{x^j}{j!}
\end{equation*}

we have

\begin{multline*}
t^{-1}\left[\gamma+\log(t)-\mathrm{Ci}(t)\right]J_{2n}(t) = \\ \frac{t^{2n+1}}{2^{2n+2}}\sum_{j=0}^{\infty}\left(\sum_{k=0}^{j}\binom{j}{k}\frac{(1)^2_k}{(2)^2_k (\frac{3}{2})_k}\frac{1}{\Gamma(j-k+2n+1)}\right)\frac{(-1)^j}{2^{2j}}\frac{t^{2j}}{j!}
\end{multline*}

Thus integrating with respect to $t$ over the half line and using \eqref{bracket rule} we obtain the bracket series

\begin{multline}
\int_{0}^{\infty}\left[\gamma+\log(t)-\mathrm{Ci}(t)\right]J_{2n}(t)\frac{\mathrm{d}}{t} = \\ \sum_{j=0}^{\infty}\left(\sum_{k=0}^{\infty}\binom{j}{k}\frac{(1)^2_k}{(2)^2_k (\frac{3}{2})_k}\frac{1}{\Gamma(j-k+2n+1)}\right)\frac{\phi_j}{2^{2j+2n+2}}\langle{2j+2n+2}\rangle 
\end{multline}

where we have used the indicator notation as well as the fact that $\sum_{k=0}^{j}\binom{j}{k}a_k = \sum_{k=0}^{\infty}\binom{j}{k}a_k$ since the binomial coefficients vanish when $k>j$.

The solution to $2j+2n+2=0$ is $j* = -n-1$, consequently the bracket series rule \eqref{bracket rule} gives

\begin{multline}
\int_{0}^{\infty}\left[\gamma+\log(t)-\mathrm{Ci}(t)\right]J_{2n}(t)\frac{\mathrm{d}}{t} = \frac{1}{2}\sum_{k=0}^{\infty}\binom{-n-1}{k}\frac{(1)^2_k}{(2)^2_k (\frac{3}{2})_k}\frac{\Gamma(n+1)}{\Gamma(n-k)}
\end{multline}

The binomial coefficients at negative integers can be rewritten as

\begin{equation*}
\binom{-n}{k} = (-1)^k \binom{n+k-1}{k}
\end{equation*}

and with this in mind we have

\begin{multline}
\frac{1}{2}\sum_{k=0}^{\infty}\binom{-n-1}{k}\frac{(1)^2_k}{(2)^2_k (\frac{3}{2})_k}\frac{\Gamma(n+1)}{\Gamma(n-k)} = \frac{1}{2}\sum_{k=0}^{n-1}\binom{n+k}{k}\frac{(-1)^k (1)^2_k}{(2)^2_k (\frac{3}{2})_k}\frac{\Gamma(n+1)}{\Gamma(n-k)}
\end{multline}

noting that the series has been truncated at $k=n-1$ due to the poles at the negative integers of the gamma function in the denominator.

Now rewriting the summand in terms of factorials, the following is obtained

\begin{align*}
\frac{1}{2}\sum_{k=0}^{n-1}\binom{n+k}{k}\frac{(-1)^k (1)^2_k}{(2)^2_k (\frac{3}{2})_k}\frac{\Gamma(n+1)}{\Gamma(n-k)} &= \sum_{j=0}^{n-1}\frac{(-1)^j 2^{2j}}{(j+1) (2j+2)!}\frac{(n+j)!}{(n-j-1)!} \\ 
&= \frac{n}{2}
{}_4 F_3\Biggl[\begin{array}{@{}c@{}c@{}c@{}c@{}c@{}c@{}r@{}r@{}r@{}r@{}}
1, & & & 1, & & & 1-n, & & & 1+n\\
& & & 2, & & &  2, & & & \tfrac{3}{2} \\
\end{array}\:\biggr\vert \, {1}\Biggr]
\end{align*}

where we have used 

$$(a)_{2n} = 2^{2n}\left(\frac{a}{2}\right)_n \left(\frac{1+a}{2}\right)_n$$

\textit{Mathematica 11} readily evaluates the finite sum as

\begin{multline*}
\sum_{j=0}^{n}\frac{(-1)^j 2^{2j}}{(j+1) (2j+2)!}\frac{(n+j)!}{(n-j-1)!} = \\ (-1)^{n+1}\frac{\Phi (-1,1,n+1)}{ 2n} + \frac{\log (2)}{2n} - \frac{(-1)^{n-1}}{4 n^2} + \frac{\psi(n+1) + \gamma}{2n} - \frac{1}{4 n^2}
\end{multline*}

where $\Phi (x,s,a)= \sum_{k=0}^{\infty}\frac{x^k}{(k+a)^s}$ is the Lerch Phi function and $\psi(x)=\Gamma'(x)/\Gamma(x)$ is the psi or digamma function. 

Noting that

\begin{equation*}
(-1)^{n+1}\frac{\Phi (-1,1,n+1)}{ 2n} + \frac{\log (2)}{2n} = \frac{A_n}{2n}
\end{equation*}

and 

\begin{equation*}
\frac{\psi(n+1) + \gamma}{2n} = \frac{H_n}{2n}
\end{equation*}

we obtain the expected result

\begin{align*}
\int_{0}^{\infty}\left[\gamma+\log(t)-\mathrm{Ci}(t)\right]J_{2n}(t)\frac{\mathrm{d}}{t} &= \frac{H_n}{2n} + \frac{A_n}{2n} - \frac{1}{4 n^2} - \frac{(-1)^{n-1}}{4 n^2} \\
&= \frac{\beta_n}{2n}
\end{align*}

We have now proven the second and third equalities of Corollary 2, namely equations \eqref{eq:C2.1} and \eqref{eq:C2.2}. In addition the first equality, equation \eqref{eq:C2.0} follows immediately from Lemma 2.

\end{proof}

\section*{Applications}

We now present some applications of the Neumann series expansions given above.  In particular, we evaluate certain integrals involving series of Bessel functions and the cosine integral in closed form.  In addition, two other results are given as further examples of applications of Propositions 1 and 2 and their corollaries.

%Corollary 3
\begin{corollary}
For $k \geq 1$,
\begin{multline}
	4\int_{0}^{\infty}\big(\gamma+\log(t)- \mathrm{Ci}(t)\big)\left(\sum_{n=1}^{\infty}\frac{J_{2n}(t)}{n^{k}}\right)\frac{\mathrm{d}t}{t} = 2\log(2)\zeta(k+1) + \zeta(k+2) \\ + \eta(k+2) + \sum_{j=1}^{k+1}\eta(k+2-j)\eta(j) - \sum_{j=1}^{k-1}\zeta(k+1-j)\zeta(j+1) \label{eq:prop3.0}
\end{multline}

\end{corollary}

\begin{proof}

	Multiplying \eqref{eq:C2.1} by ${4}/{n^k}$ and summing over $[1,\infty)$ gives
	
	\begin{equation*}
		4\sum_{n=1}^{\infty}\frac{1}{n^{k}}\int_{0}^{\infty}\left[\gamma+\log(t)- \mathrm{Ci}(t)\right]J_{2n}(t)\frac{\mathrm{d}t}{t} = 2\sum_{n=1}^{\infty}\frac{\beta_n}{n^{k+1}}
	\end{equation*}
	
	Expressing the coefficients $\beta_n$ in the slightly altered form
	
	\begin{equation*}
		\beta_n = H_{n-1} + A_{n-1} + \frac{1}{2n} + \frac{(-1)^{n-1}}{2n}
	\end{equation*}
	
	then breaking up the sum on the RHS into its constituent parts we have
	
	\begin{equation}
		2\sum_{n=1}^{\infty}\frac{\beta_n}{n^{k+1}} = 2\sum_{n=1}^{\infty}\frac{H_{n-1}}{n^{k+1}} + 2\sum_{n=1}^{\infty}\frac{A_{n-1}}{n^{k+1}} + \zeta(k+2) + \eta(k+2) \label{eq:prfP3.0}
	\end{equation}
	
	Evaluations of Euler sums of the type found above are well known, for instance in \cite{borwein} we find the following evaluation due to Euler
	
	\begin{equation*}
		2\sum_{n=1}^{\infty}\frac{H_{n-1}}{n^k} = k\zeta(k+1) - \sum_{j=1}^{k-2}\zeta(k-j)\zeta(j+1)
	\end{equation*}
	
	as well as Nielson's formula
	
	\begin{equation*}
		2\sum_{n=1}^{\infty}\frac{A_{n-1}}{n^k} = 2\log(2)\zeta(k) - k\zeta(k+1) + \sum_{j=1}^{k}\eta(k+1-j)\eta(j)
	\end{equation*}
	
	Apparently after inserting these equations into \eqref{eq:prfP3.0} above, we are led to
	
	\begin{multline*}
		4\sum_{n=1}^{\infty}\frac{1}{n^{k}}\int_{0}^{\infty}\left[\gamma+\log(t)- \mathrm{Ci}(t)\right]J_{2n}(t)\frac{\mathrm{d}t}{t} = 2\log(2)\zeta(k+1) + \zeta(k+2) \\ + \eta(k+2) + \sum_{j=1}^{k+1}\eta(k+2-j)\eta(j) - \sum_{j=1}^{k-1}\zeta(k+1-j)\zeta(j+1)
	\end{multline*}
	
	which after results in \eqref{eq:prop3.0} after changing the order of summation and integration.
	
\end{proof}

%Example of Corollary 3
\begin{example}

\begin{equation}
\int_{0}^{\infty}\left(\frac{\gamma+\log(t)-\mathrm{Ci}(t)}{t}\right)\left(2S_{-1,0}(t)-\frac{\partial^2}{\partial v^2}J_v(t)\big\vert_{v=0}\right)\mathrm{d}t
\nonumber = \frac{7}{8}\log(2)\zeta(3) \label{eq:C3.2}
\end{equation}

\end{example}

\begin{proof}

With the help of the following formula (57.2.38 in \cite{hansen}) 

\begin{equation*}
\sum_{n=1}^{\infty}\frac{1}{n^2}J_{2n}(t) = 2S_{-1,0}(t)-\frac{\pi^2}{3}J_0(t)-\frac{\partial^2}{\partial v^2}J_v(t)\big\vert_{v=0}
\end{equation*}

where $S_{\mu,v}(t)$ (see \cite{GR}) is the associated Lommel function, we find that setting $k = 2$ in Corollary 3 gives

\begin{multline}
4\int_{0}^{\infty}\left(\frac{\gamma+\log(t)-\mathrm{Ci}(t)}{t}\right)\left(2S_{-1,0}(t)-\frac{\pi^2}{3}J_0(t)-\frac{\partial^2}{\partial v^2}J_v(t)\big\vert_{v=0}\right)\mathrm{d}t \\ = \frac{7}{2}\log(2)\zeta(3)\label{eq:PrfC3}
\end{multline}

We may separate the integral into two parts, namely

\begin{multline}
4\int_{0}^{\infty}\left(\frac{\gamma+\log(t)-\mathrm{Ci}(t)}{t}\right)\left(2S_{-1,0}(t)-\frac{\partial^2}{\partial v^2}J_v(t)\big\vert_{v=0}\right)\mathrm{d}t \\ \qquad - \frac{4\pi^2}{3}\int_{0}^{\infty}\left(\frac{\gamma+\log(t)-\mathrm{Ci}(t)}{t}\right)J_0(t)\mathrm{d}t \\ = \frac{7}{2}\log(2)\zeta(3)
\end{multline}

where the first integral cannot be split up any further. Then noting that

\begin{equation*}
\int_{0}^{\infty}\left(\frac{\gamma+\log(t)-\mathrm{Ci}(t)}{t}\right)J_0(t)\mathrm{d}t =
0
\end{equation*}

by taking the limit as $n \rightarrow 0$ in equation \eqref{eq:C2.0} of Corollary 2.  After dividing each side of \eqref{eq:PrfC3} by $4$ gives us our desired result.

\end{proof}

\begin{remark}
There are in fact, known closed form solutions for the Neumann series

\begin{equation*}
\sum_{n=1}^{\infty}\frac{(\pm 1)^n}{n^{2k}}J_{2n}(t)
\end{equation*}

expressed in terms of recursively defined functions.  In particular, we have 

\begin{equation*}
\sum_{n=1}^{\infty}\frac{1}{n^{2k+2}}J_{2n}(t) = T_k(t)
\end{equation*}

where $T_k(t)$ is defined for $k=0$ as

\begin{equation*}
T_0(t) = 2S_{-1,0}(t)-\frac{\pi^2}{3}J_0(t)-\frac{\partial^2}{\partial v^2}J_v(t)\big\vert_{v=0}
\end{equation*}

and for $k>0$ as

\begin{multline*}
\sum_{j=0}^{k-1}\frac{(-1)^j \pi^{2k-2j-1}}{(2k-2j-1)!}T_{j}(t) \\ = \frac{(-1)^{k+1} 2^{2k}}{(2k)!}\frac{\partial^{2k}}{\partial v^{2k}}\left[v\sin\left(\frac{v\pi}{2}\right)S_{-1,v}(t)\right]_{v=0} - \frac{\pi^{2k+1}}{2(2n+1)!}J_0(t)
\end{multline*}

Thus, from equation \eqref{eq:C2.0} of Corollary 2 and Corollary 3 we see that

\begin{align*}
	4\int_{0}^{\infty}\big(\gamma+\log(t)- \mathrm{Ci}(t)\big)T_k(t)\frac{\mathrm{d}t}{t} &= 2\int_{0}^{\pi/2}\left[\zeta(2k+3)-\mathrm{Cl}_{2k+3}(2t)\right]\cot(t)\mathrm{d}t \\ &= 2\log(2)\zeta(2k+3) + \zeta(2k+4) \\ & \quad + \eta(2k+4) + \sum_{j=1}^{2k+3}\eta(2k+4-j)\eta(j) \\ & \quad - \sum_{j=1}^{2k+1}\zeta(2k+3-j)\zeta(j+1) 
\end{align*}

where $\mathrm{Cl}_{k}(t)$ is the Clausen function defined by 

\begin{equation*}
\mathrm{Cl}_{2k}(t) = \sum_{n=1}^{\infty}\frac{\sin(nt)}{n^{2k}}
\end{equation*}

\begin{equation*}
\mathrm{Cl}_{2k+1}(t) = \sum_{n=1}^{\infty}\frac{\cos(nt)}{n^{2k+1}}
\end{equation*}

As an example, Corollary 3 is equivalent to

\begin{equation*}
\int_{0}^{\pi/2}\left[\zeta(3)-\mathrm{Cl}_3(2t)\right]\cot(t)\mathrm{d}t = \frac{7}{4}\log(2)\zeta(3)
\end{equation*}

\end{remark}

%Corollary 4
\begin{corollary}
For $k \geq 1$,

\begin{multline}
4\int_{0}^{\infty}\big(\gamma+\log(t)- \mathrm{Ci}(t)\big)\left(\sum_{n=1}^{\infty}\frac{(-1)^{n}}{n^{2k-1}}J_{2n}(t)\right)\frac{\mathrm{d}t}{t} =  \zeta(2k+1) + \eta(2k+1) \\ - 2\eta(1)\eta(2k) + 2\sum_{j=1}^{k-1}\zeta(2k+1-2j)\eta(2j) - 2\sum_{j=1}^{k}\eta(2k+1-2j)\zeta(2j) \label{eq:prop4.0}
\end{multline}

\end{corollary}

\begin{proof}

Proceeding in a similar manner as the previous proof, this time multiplying \eqref{eq:C2.1} by ${4(-1)^{n}}/{n^{2k-1}}$ and summing over $[1,\infty)$ we obtain
		
	\begin{equation*}
		4\sum_{n=1}^{\infty}\frac{(-1)^{n}}{n^{2k-1}}\int_{0}^{\infty}\left[\gamma+\log(t)- \mathrm{Ci}(t)\right]J_{2n}(t)\frac{\mathrm{d}t}{t} = 2\sum_{n=1}^{\infty}\frac{(-1)^{n}\beta_n}{n^{2k}}
	\end{equation*}
	
	The reason for restricting our multiplying by odd powers of $(-1)^{n-1}/n$ stems from the fact that---for the most part---closed forms in terms of eta and zeta functions exist only for alternating Euler sums of odd weight.  In our case, this translates to sums of the form
	
	\begin{equation*}
		\sum_{n=1}^{\infty}\frac{(-1)^{n}H_{n-1}}{n^{2k}} \quad \text{and} \quad \sum_{n=1}^{\infty}\frac{(-1)^{n}A_{n-1}}{n^{2k}} 
	\end{equation*}   

	In particular, we have the formulas originally due to Sitaramachandrarao \cite{rao}
	
	\begin{multline*}
		2\sum_{n=1}^{\infty}\frac{(-1)^{n}H_{n-1}}{n^{2k}} = \zeta(2k+1) - (2k-1)\eta(2k+1) \\ + 2\sum_{j=1}^{k-1}\zeta(2k+1-2j)\eta(2j) 
	\end{multline*}

	\begin{multline*}
		2\sum_{n=1}^{\infty}\frac{(-1)^{n}A_{n-1}}{n^{2k}} = \zeta(2k+1) + (2k+1)\eta(2k+1) - 2\eta(1)\eta(2k) \\ - 2\sum_{j=1}^{k}\eta(2k+1-2j)\zeta(2j) 
	\end{multline*}
	
	In addition we have, 
	
	\begin{equation*}
		2\sum_{n=1}^{\infty}\frac{(-1)^n\beta_n}{n^{2k}} = 2\sum_{n=1}^{\infty}\frac{(-1)^{n}H_{n-1}}{n^{2k}} + 2\sum_{n=1}^{\infty}\frac{(-1)^{n}A_{n-1}}{n^{2k}} - \zeta(2k+1) - \eta(2k+1) 
	\end{equation*}
	
	Therefore, plugging in Sitaramachandrarao's formulas we arrive at
	
	\begin{multline*}
		2\sum_{n=1}^{\infty}\frac{(-1)^{n}\beta_n}{n^{2k}} = \zeta(2k+1) + \eta(2k+1) - 2\eta(1)\eta(2k) \\  + 2\sum_{j=1}^{k-1}\zeta(2k+1-2j)\eta(2j) - 2\sum_{j=1}^{k}\eta(2k+1-2j)\zeta(2j)
	\end{multline*}
	
	which proves Corollary 4 after once again interchanging the order of summation and integration.
\end{proof}

\begin{example}

\begin{multline}
\int_{0}^{\infty}\left(\frac{\gamma+\log(t)-\mathrm{Ci}(t)}{t}\right)\left(\frac{\pi}{2}Y_0(t)-\log\frac{t}{2}J_0(t)\right)\mathrm{d}t \\ =
\frac{\pi^2}{4}\log(2) - \frac{7}{8}\zeta(3) \label{eq:C3.0}
\end{multline}

where $Y_v(t)$ is the Bessel function of the second kind.

\end{example}

\begin{proof}
Setting $k = 1$ in Corollary 4 gives

\begin{equation*}
4\int_{0}^{\infty}\big(\gamma+\log(t)- \mathrm{Ci}(t)\big)\left(\sum_{n=1}^{\infty}\frac{(-1)^{n}}{n}J_{2n}(t)\right)\frac{\mathrm{d}t}{t} = - \frac{\pi^2}{2}\log(2) + \frac{7}{4}\zeta(3)
\end{equation*}

Conveniently, the series on the LHS can be found in \cite{GR} as formula 8.515.7

\begin{equation*}
\sum_{n=1}^{\infty}\frac{(-1)^{n}}{2n}J_{2n}(t) = -\frac{\pi}{8}Y_0(t)+\frac{1}{4}\left(\log\frac{t}{2} + \gamma\right)J_0(t)
\end{equation*}

Thus we have the following integral

\begin{multline*}
\int_{0}^{\infty}\big(\gamma+\log(t)-\mathrm{Ci}(t)\big)\left(-\frac{\pi}{2}Y_0(t) + \left(\log\frac{t}{2} + \gamma\right)J_0(t)\right)\frac{\mathrm{d}t}{t} = \\ - \frac{\pi^2}{2}\log(2) + \frac{7}{4}\zeta(3)
\end{multline*}

Again splitting up the integrand and utilizing

\begin{equation*}
\int_{0}^{\infty}\left(\frac{\gamma+\log(t)-\mathrm{Ci}(t)}{t}\right)J_0(t)\mathrm{d}t =
0
\end{equation*}
 
we arrive at our desired conclusion

\begin{multline*}
\int_{0}^{\infty}\left(\frac{\gamma+\log(t)-\mathrm{Ci}(t)}{t}\right)\left(-\frac{\pi}{2}Y_0(t)+\log\frac{t}{2}J_0(t)\right)\mathrm{d}t =
-\frac{\pi^2}{4}\log(2) + \frac{7}{8}\zeta(3)
\end{multline*}

\end{proof}

The next result gives an integral representation of another Neumann series involving the coefficients $\beta_n$.

\begin{corollary}
For $a \in \mathbb{R}$, 
\begin{equation}
\sum_{n=1}^{\infty}\frac{(-1)^n}{n} J_{2n}(a)\beta_n = 2\int_{0}^{\infty}\left[\gamma+\log(t)-\mathrm{Ci}(t)\right]J_{0}\left(\sqrt{a^2 + t^2}\right)\frac{\mathrm{d}t}{t}
\end{equation}

\end{corollary}

\begin{proof}
The proof of Corollary 5 begins with the identity

\begin{equation}
J_{0}\left(\sqrt{a^2 + t^2}\right) - J_0(a)J_0(t) = \sum_{n=1}^{\infty}(-1)^n J_{2n}(a)J_{2n}(t) \label{eq:C5.0}
\end{equation}

which can be obtained by setting $\phi = \pi/2$ in Nuemann's addition theorem (8.531.1 in \cite{GR})

\begin{equation*}
J_{0}\left(\sqrt{x^2 + y^2-2xy\cos(\phi)}\right) =  J_0(x)J_0(y) + \sum_{n=1}^{\infty}J_{n}(x)J_{n}(y)\cos(n\phi)
\end{equation*}

Multiplying \eqref{eq:C5.0} by $\left(\gamma+\log(t)-\mathrm{Ci}(t)\right)t^{-1}$ and integrating with respect to $t$ proves the result after once again using \eqref{eq:C2.1} of Corollary 3 and the fact that the second term on the LHS vanishes under integration.

\end{proof}

The final result gives a sine integral analogue of the examples of Corollaries 3 and 4 given above.

\begin{corollary}

\begin{equation*}
\int_{0}^{\infty}\mathrm{Si}(t)\left(\log\frac{t}{2}J_1(t)-\frac{\pi}{2}Y_1(t) - \frac{1}{t}J_0(t)\right)\frac{\mathrm{d}t}{t} = 4-4G-\gamma
\end{equation*}

where $G = \sum_{n=0}^{\infty}\frac{(-1)^n}{(2n+1)^2}\approx 0.9159655...$ is Catalan's constant.

\end{corollary}

\begin{proof}
The proof of Corollary 2 begins by multiplying \eqref{eq:C1.1} by $(-1)^n\frac{2n+1}{n(n+1)}$ and summing over $[1,\infty)$.  This gives

\begin{multline}
\int_{0}^{\infty}\mathrm{Si}(t)\left(\left(\log\frac{t}{2} + \gamma -1 \right)J_1(t)-\frac{\pi}{2}Y_1(t) - \frac{1}{t}J_0(t)\right)\frac{\mathrm{d}t}{t} = \sum_{n=1}^{\infty}\frac{(-1)^{n}\alpha_n}{n(n+1)} \label{eq:C6.0}
\end{multline}

where we have interchanged the order of summation and integration on the LHS and used the identity

\begin{equation*}
\sum_{n=1}^{\infty}(-1)^{n}\frac{2n+1}{n(n+1)}J_{2n+1}(t) = \left(\log\frac{t}{2} + \gamma -1 \right)J_1(t)-\frac{\pi}{2}Y_1(t) - \frac{1}{t}J_0(t)
\end{equation*}

found as equation 8.514.9 in \cite{GR}.  We mention that the corresponding formula given in \cite{hansen} contains a typographical error, lacking the $-1$ in the parenthetical term on the RHS.

The coefficients of the series in \eqref{eq:C6.0} are expanded as

\begin{align*}
	\sum_{n=1}^{\infty}\frac{(-1)^{n}\alpha_n}{n(n+1)} &= 2\sum_{n=1}^{\infty}\frac{(-1)^{n}}{n(n+1)}\sum_{k=1}^{n}\frac{(-1)^{k-1}}{2k-1} + \sum_{n=1}^{\infty}\frac{1}{n(n+1)(2n+1)} 
\end{align*}

and after a number of series manipulations and two applications of the identity

\begin{equation*}
\sum_{n=1}^{\infty}\frac{(-1)^n}{n}\sum_{k=1}^{n}\frac{(-1)^{k-1}}{2k-1} = -G
\end{equation*}

we arrive at

\begin{equation*}
\int_{0}^{\infty}\mathrm{Si}(t)\left(\left(\log\frac{t}{2} + \gamma -1 \right)J_1(t)-\frac{\pi}{2}Y_1(t) - \frac{1}{t}J_0(t)\right)\frac{\mathrm{d}t}{t} = 3-4G
\end{equation*}

Remembering that Corollary 1 gives

\begin{equation*}
\int_{0}^{\infty}\mathrm{Si}(t)J_{1}(t)\frac{\mathrm{d}t}{t} = 1
\end{equation*}

we may split up the integral into two parts as we have done before to obtain

\begin{equation*}
\int_{0}^{\infty}\mathrm{Si}(t)\left(\log\frac{t}{2}J_1(t)-\frac{\pi}{2}Y_1(t) - \frac{1}{t}J_0(t)\right)\frac{\mathrm{d}t}{t} + \gamma - 1 = 3-4G
\end{equation*}

The proof is completed after subtracting $\gamma -1$ from both sides.
\end{proof}

MSC 2010: 41A58, 33B99, 33C10\\

University of Wisconsin-Madison; Madison, Wisconsin. [Student]\\

\textit{E-mail: sanfordchance@gmail.com}

\end{document}